\documentclass[12pt]{article}

\usepackage{mathtools}
\mathtoolsset{showonlyrefs}
\usepackage{float}
\usepackage{amsmath,amssymb,amsthm}
\usepackage[american]{babel}
\usepackage[latin1]{inputenc}
\usepackage[numbers]{natbib}
\usepackage{mathrsfs}
\usepackage{booktabs}
\usepackage{url}
\usepackage{amscd}
\usepackage{amstext}
\usepackage{amsfonts}
\usepackage{bbm}
\usepackage{soul}
\usepackage[latin1]{inputenc}
\usepackage[T1]{fontenc}
\usepackage{enumerate}
\numberwithin{equation}{section}

\newcommand{\mathset}[1]{\mathbbm{#1}}
\newcommand{\morf}[4][\to]{ #2 \colon #3 #1 #4}
\newcommand{\F}{\mathcal{F}}
\newcommand{\ft}{(\F_t)_{t\geq 0}}

\newcommand{\Fi}{\mathbb{F}}

\newcommand{\V}{\mathcal{V}}

\newcommand{\A}{\textbf{A}}

\newcommand{\s}{\mathscr S}

\newcommand{\B}{\mathscr B}
\newcommand{\N}{\mathset{N}}

\newcommand{\ca}{c\`adl\`ag}
\newcommand{\R}{\mathset{R}}

\newcommand{\abs}[2][]{#1\lvert #2#1\rvert}

\newcommand{\p}[2][]{#1 ( #2_t #1)_{t\geq 0}}

\renewcommand{\B}{\mathscr B}

 \def \lt{[\hskip-1.3pt[ }
\def \rt{ ]\hskip-1.3pt]}

\def\E{\mathbb{E}}
\def\P{\mathbb{P}}
\def\1{\mathbf{1}}
\def\({\left(}
\def\){\right)}

\def\X{\mathbf{X}}
\def\M{\mathbf{M}}
\def\Y{\mathbf{Y}}

\def\citemma{\cite[Proposition~3.5]{fv-mma}}

\theoremstyle{plain}
\newtheorem{lemma}{Lemma}[section]
\newtheorem{proposition}[lemma]{Proposition}
\newtheorem{theorem}[lemma]{Theorem}
\newtheorem{corollary}[lemma]{Corollary}

\theoremstyle{definition}
\newtheorem{example}[lemma]{Example}

\theoremstyle{definition}

\newtheorem{remark}[lemma]{Remark}

\newcommand{\devnull}[1]{}

\title{Structure of infinitely divisible semimartingales}
\author{Andreas Basse-O'Connor$^{\dagger}$ and Jan Rosi\'nski$^{\ddagger}$\\
{\normalsize Aarhus University and University of Tennessee }\\
{\normalsize $^\dagger$E-mail: basse@imf.au.dk\qquad $^\ddagger$E-mail: rosinski@math.utk.edu}
}

\begin{document}\maketitle
\begin{abstract}
This paper gives a complete characterization of infinitely divisible semimartingales, i.e., semimartingales whose finite dimensional distributions are infinitely divisible. An explicit and essentially unique decomposition of such semimartingales is obtained. A new approach, combining series decompositions of infinitely divisible processes with detailed analysis of their jumps, is presented. 
As an ilustration of the main result, the semimartingale property is explicitely determined for a large class of stationary increment processes and several examples of processes of interest are considered. 
These results extend Stricker's theorem characterizing Gaussian semimartingales and Knight's theorem describing Gaussian moving average semimartingales, in particular.


\end{abstract}

\section{Introduction}

 Semimartingales play a crucial role in stochastic analysis as they form the class of   \emph{good integrators}  for the It\^o type stochastic integral,  cf.\ the  Bichteler--Dellacherie Theorem \cite{Bichteler}, see also \citet{B-D-direct-proof} for a direct proof. 
Semimartingales also play a fundamental role in mathematical finance. Roughly speaking, the (discounted) asset price process  must be a semimartingale in order to preclude arbitrage opportunities, see   
\citet[Theorems 1.4, 1.6]{B-D-direct-proof} for details, see also \cite{K-P}.
The question whether a given process is a semimartingale is also of importance in stochastic modeling, where long memory processes with possible jumps and high volatility are considered as driving processes for stochastic differential equations. Examples of such processes include various fractional, or more generally, Volterra processes driven by L\'evy processes. 

The problem of identifying semimartingales within given classes of stochastic processes has a long history. In the context of Gaussian processes, it was intensively studied in 1980s. \citet{Galchuk} investigated Gaussian semimartingales addressing a question posed by Prof. A.N.\ Shiryayev. Key results  on Gaussian semimartingales are due to \citet[Theorem~1.8]{Gau_Qua} and \citet[Th\'eor\`eme~1]{Stricker_gl}, see also   
 \citet[Ch.~4.9]{Lip_S} and \cite{Emery,  Andreas2, Andreas1, Galchuk}. A particularly interesting result is due to  \citet[Theorem~6.5]{Knight}: Let $\X=(X_t)_{t\geq 0}$  
 be  a  Gaussian moving average of the form 
\begin{equation}\label{MA-X}
 X_t=\int_{-\infty}^t \phi(t-s)\, dW_s, 
\end{equation}
where $(W_t)_{t\in \R}$ is a Brownian motion and $\morf{\phi}{\R}{\R}$ is a Lebesgue square integrable deterministic function vanishing on the negative axis. Then  $\X$ is a semimartingale with respect to the filtration $(\F^W_t)_{t\geq 0}$ if and only if $\phi$ is absolutely continuous on $\R_+$ with a square integrable derivative. Then, $\X$ can be decomposed uniquely into a  Brownian motion and a predictable process of finite variation. Extensions of Knight's result were given by \citet{Yor:sem},   see also \cite{Cherny, Patrick, Andreas3, Basse_Graversen}.

The class of infinitely divisible processes is a natural extension of the class of Gaussian processes. A process $\mathbf{X}=(X_t)_{t\geq 0}$ is said to be infinitely divisible if its all finite dimensional distributions are infinitely divisible.  This work is aimed to determine the structure of infinitely divisible semimartingales. 

Recall that a semimartingale $\mathbf{X}=(X_t)_{t\geq 0}$, by definition, admits a decomposition 
 \begin{equation}\label{decomp-X}
 X_t=X_0+M_t+A_t,\quad t\geq 0,
 \end{equation}
 where $\M= \p M$ is a local martingale and $\A= \p A$ is a process of finite variation.  
 The problem of identifying infinitely divisible semimartingales can be divided into two questions for an infinitely divisible process $\mathbf{X}=(X_t)_{t\geq 0}$:
 \begin{enumerate}
\item[Q1.] 
Assuming that $\X$ is a semimartingale, what  is the structure of the processes $\bf M$ and $\bf A$ in its decomposition \eqref{decomp-X}? In particular, do they have to be infinitely divisible?
\item[Q2.]   How to verify whether or not a given infinitely divisible process $\X$ is a semimartingale?
\end{enumerate}

The family of infinitely divisible processes constitutes a huge class, so we will focus on a parametrization that will be convenient to state our results as well as for applications. 
 We will assume that a primary description of an infinitely divisible processes $\X=(X_t)_{t\geq 0}$ is its 'spectral' representation of the form
\begin{equation}\label{rep-X-in}
 X_t=\int_{(-\infty,t]\times V} \phi(t,s,v)\,\Lambda(ds,dv),   
\end{equation} 
where $V$ is a set, $\Lambda$ is an  independently scattered infinitely divisible random measure on $\R\times V$, and $\morf{\phi}{\R^2\times V}{\R}$ is a measurable deterministic function (see Section \ref{sec-def} for details). If we further assume that $\phi(t,s,x)=0$ whenever $s>t$, then $\X$ is adapted to the filtration $\Fi^{\Lambda}=(\mathcal{F}^{\Lambda}_t)_{t \ge 0}$, where $\mathcal{F}^{\Lambda}_t$ is the $\sigma$-algebra generated by $\Lambda$ restricted to $(-\infty, t]\times V$. We characterize the semimartingale property of $\X$ with respect to  $\Fi^{\Lambda}$, which,  in the spirit, is similar to the above mentioned results of Stricker and Knight. However, in the non Gaussian case, our methods are completely different. We combine series representations of \ca\ infinitely divisible processes with  detailed analysis of their jumps, which is a novel approach as far as we know. This is possible because such series representations converge uniformly a.s.\ on compacts, as shown in \citet[Theorem 3.1]{Ito-Nisio-D}.

Section~\ref{sec-def} contains preliminary definitions and facts. Our main result, Theorem~\ref{thm1}, is stated and proved in Section~\ref{s-sem}. It gives the necessary and sufficient conditions for $(\X, \Fi^{\Lambda})$ to be a semimartingale, together with an essentially unique explicit decomposition of $\X$ into infinitely divisible components. It completely answers the above question Q1 and gives a framework how to approach question Q2 in concrete situations.  In Section~\ref{appl} we obtain explicit necessary and sufficient conditions for a large class of stationary increment infinitely divisible processes to be  semimartingales. These conditions generalize, in a natural way,   findings in \citet{Basse_Pedersen}. We conclude this paper with  examples showing how these conditions can be verified for several processes of interest.

\section{ Preliminaries}\label{sec-def}

Throughout the paper  $(\Omega,\F,\P)$ stands for a probability space and $(V, \mathcal V)$ denotes a  countably generated measurable space.
 Consider a process $\X=(X_t)_{t\geq 0}$ of the form 
\begin{equation}\label{def-X}
 X_t=\int_{(-\infty,t]\times V} \phi(t, s,v)\,\Lambda(ds,dv)  
\end{equation}
where $\morf{\phi}{\R^2\times V}{\R}$ is a measurable deterministic function such that for every $(s,v) \in \R \times V$, $\phi(\cdot, s,v)$ is \ca\ and   $\Lambda$ is an independently scattered infinitely divisible random measure   
on  a $\sigma$-ring $\s$   of subsets of $\R\times V$ such that
\begin{equation} \label{V}
\big\{[a,b]\times B: \ [a,b] \subset \R,  \ B \in \mathcal{V}_0\big\} \subset \s \subset \B(\R)\otimes \V,
\end{equation}
for some countable ring $\mathcal{V}_0$ generating $\mathcal{V}$. It follows that $\sigma(\s)=\B(\R)\otimes \V$. For example, $\Lambda$ can be a Poisson random measure. 
In general, we assume that for every $A \in \s$, $\Lambda(A)$ has an infinitely divisible distribution with the cumulant
\begin{equation}\label{Lambda} 
 \log \E e^{i\theta\Lambda(A)}= \int_{A} \Big[ i\theta b(u)-\frac{1}{2}\theta^2 \sigma^2(u) +\int_{\R} \big(e^{i\theta x}-1-i\theta \lt x\rt\big) \, \rho_{u}(dx) \Big] 
 \,\kappa(d u),
\end{equation}
where $u = (s,v) \in \R\times V$, $\morf{b}{\R\times V}{\R}$ is a measurable function,  $\kappa$ is a $\sigma$-finite measure on $\R\times V$,  $\{\rho_{u}\}_{u\in \R\times V}$ is a measurable family of L\'evy measures on $\R$, and 
\begin{equation}\label{tr}
 \lt x\rt= \frac{x}{|x| \vee 1} = \begin{cases}
    x    &  \text{if } |x|<1\, ,  \\ 
    \mathrm{sgn}(x)    & \text{otherwise}
\end{cases}  
\end{equation}
is a truncation function. For example, if $\Lambda$ is a Poisson  random measure, then $b(u)=1$, $\sigma^2(u)=0$, and $\rho_u=\delta_1$. The integral in \eqref{def-X} is defined as in \citet{Rosinski_spec}.   Accordingly to \cite[Theorem 2.7]{Rosinski_spec},  given a measurable deterministic function $f$, the integral $\int_{\R\times V} f(u) \, \Lambda(du)$ exists	if and only if
\begin{enumerate}[(a)]
   \item \label{i1} $\int_{\R\times V} |B(f(u), u)| \, \kappa(du) < \infty$,
      \item \label{i2} $\int_{\R\times V} K(f(u), u) \, \kappa(du) < \infty$,
\end{enumerate} 
where
\begin{align} \label{B}
B(x, u) = {}& x b(u) + \int_{\R} \big( \lt xy\rt - x \lt y\rt\big) \, \rho_{u}(dy)\quad \text{and}
\\
\label{K}
K(x, u) = {}& x^2 \sigma^2(u) + \int_{\R}  \lt xy\rt^{2} \, \rho_{u}(dy), \qquad x \in \R, \ u \in \R\times V.
\end{align}
The process $\X$ in \eqref{def-X} is infinitely divisible, i.e., all its finite dimensional distributions are infinitely divisible.


We will further require that 
\begin{equation} \label{Lambda2}
\Lambda(\{s\} \times B) = 0\quad  \text{a.s.\ for every } s \in \R, \ B \in \mathcal{V}_0.
\end{equation}
This condition is equivalent to that $\kappa(\{s\} \times V)=0$ for every $s \in \R$. 

Let $\Fi^{\Lambda}=(\mathcal{F}^{\Lambda}_t)_{t \ge 0}$ be the augmented  filtration generated by $\Lambda$,  i.e.\ $\Fi^\Lambda$ is the least filtration satisfying the \emph{usual conditions} of right-continuity and completeness such that 
\begin{equation}\label{def_fil}
\sigma\Big(\Lambda(A): A\in \s,\, A\subseteq (-\infty,t]\times V\Big)\subseteq \F^\Lambda_t,\qquad t\geq 0.
 \end{equation} 

\noindent {\sc Definition} (semimartingale):
 Let  $\Fi=\ft$ be a  filtration. An $\Fi$-adapted process $\X=\p X$ is  a semimartingale with respect to $\Fi$ if it admits a decomposition 
\begin{equation}\label{ca_de_4}
 X_t=X_0+A_t+M_t,\qquad t\geq 0,
 \end{equation}
where $\M=\p M$ is a \ca\ local martingale with respect to $\Fi$,  $\A =\p A$ is  an $\Fi$-adapted process with \ca\ paths of finite variation and $A_0=M_0=0$. (C\`adl\`ag  means right-continuous with left-hand limits.) Moreover, $\X$ is called  a \emph{special semimartingale} if, in addition,  $\bf A$ in \eqref{ca_de_4} can be chosen predictable. In the later case, the decomposition \eqref{ca_de_4} is unique and  is called the \emph{canonical decomposition} of $\X$ and processes $\M$ and $\A$ are called the martingale and finite variation components, respectively. 
We refer to \citet{Jacod_S} and \citet{Protter} for basic properties of semimartingales.  

For stochastic processes $\X=\p X$ and $\Y=\p Y$   we will write $\X=\Y$ when $\X$ and $\Y$ are indistinguishable, i.e., $\{\omega: X_t(\omega) \ne Y_t(\omega) \ \text{for some} \ t \ge 0\}$ is a $\P$-null set. 
For each \ca\ function $\morf{g}{\R_+}{\R}$  let $\Delta g(t)=\lim_{s\uparrow t, s<t} (g(t)-g(s))$ denote the jump of $g$ at $t>0$ and $\Delta g(0)=0$.

 \section{Infinitely divisible semimartingales} \label{s-sem}

The following is our main result.
\begin{theorem}\label{thm1}
Let $\mathbf{X}=\p X$ be a process of the form 
\begin{equation}\label{def_Y_phi} 
X_t=\int_{(-\infty, t] \times V} \phi(t, s,v)\,\Lambda(ds,dv), 
\end{equation}
specified by \eqref{def-X}--\eqref{tr}, and let $B$ be given by \eqref{B}. Assume that  for every $t>0$ 
\begin{equation} \label{drift}
\int_{(0,t] \times V} \big| B\big(\phi(s, s, v), (s,v)\big) \big| \, \kappa(ds, dv) < \infty.
\end{equation}
Then  $\mathbf{X}$ is a semimartingale with respect to the filtration 
$\mathbb{F}^{\Lambda}=(\mathcal{F}^{\Lambda}_t)_{t \ge 0}$ if and only if
\begin{equation}\label{dec}
   X_t = X_0 + M_t + A_t, \quad t \ge 0,
\end{equation}
where $\M= \p M$ is a  continuous in probability semimartingale with  independent increments given by 
\begin{equation} \label{M}
M_t = \int_{(0,t] \times V} \phi(s,s,v)\,\Lambda(ds,dv), \quad t\ge 0
\end{equation}
(i.e., the integral exists), and $\mathbf{A}= \p A$ is a predictable \ca\ process of finite variation  given by 
\begin{equation} \label{A}
A_t = \int_{(-\infty,t] \times V} [\phi(t,s,v) -  \phi(s_{+}, s,v)] \,\Lambda(ds,dv).
\end{equation}
Decomposition \eqref{dec}	 is unique in the following sense: If $\X=X_0+\M'+\A'$, where $\M'$ is a continuous in probability semimartingale with independent increments and $\A'$ is a predictable \ca\ process of finite variation, then $\M'=\M+g$ and $\A'=\A-g$ for some  continuous deterministic function $g$ of finite variation, with $\M$ and $\A$  given by \eqref{M} and \eqref{A}.

 If  $\bf X$ is a semimartingale, then it is a special semimartingale if and only if   $\E\abs{M_t}<\infty$ for all $t>0$; and in this case  $(M_t-\E M_t)_{t\geq 0}$ is a  martingale and  
\begin{equation}
  X_t= X_0+ (M_t -\E M_t)+ (A_t+\E M_t),\quad t\geq 0
\end{equation}
  is the canonical decomposition of $\X$. 
\end{theorem}
\bigskip

\begin{remark}
{ \rm
Condition~\eqref{drift} is always satisfied when $\Lambda$ is symmetric. Indeed, in this case $B=0$. 
}
\end{remark}

\begin{remark}
{ \rm
Stricker's theorem \cite{Stricker_gl} says that Gaussian semimartingales are those who admit a decomposition into a Gaussian martingale with independent increments and a predictable process of finite variation. Our theorem is an exact extension of this result to the infinitely divisible case.
}
\end{remark}

\begin{remark}
{ \rm
There is a slight inconsistency in the notations used in \eqref{decomp-X} and \eqref{dec}. 
In \eqref{dec} $\M$ is a semimartingale with independent increments, which is  a martingale when centered, in which case \eqref{decomp-X} and \eqref{dec} coincide. However, if $\M$ does not have zero expectation, then further decomposition of $\M$ into a martingale and a process of finite variation is needed to obtain \eqref{decomp-X}, but this is standard, see e.g.\ \cite[Theorem~19.2]{Sato}. 
}
\end{remark}

\begin{example}
Consider the setting in Theorem~\ref{thm1} and suppose that  $\Lambda$ is an  $\alpha$-stable random measure and $\alpha\in (0,1)$.  Then $\bf X$ is a semimartingale with respect to $\Fi^\Lambda$ if and only if it is  of finite variation. This follows by Theorem~\ref{thm1} because the process $\bf M$ given by \eqref{M-levy} is of finite variation. 

Indeed,    the  L\'evy--It\^o decomposition, see  \cite[Theorem~19.3]{Sato}, yields that  $\M=\bar \M+\mathbf{a}$, where $\bf a$ is a continuous deterministic function and process $\bar \M$ is of finite variation, see \cite[Eq.~(20.24)]{Sato}. Moreover since $\M$ is a semimartingale, the function $\bf a$ is of finite variation, cf.\ \cite[Ch.~II, Corollary~5.11]{Jacod_S}, which implies that $\M$ is of finite variation. 
%
%
%
%
\end{example}

\medskip

\begin{proof}[Proof of Theorem~\ref{thm1}]

The sufficiency is obvious. To show the necessary part we need to show that a semimartingale $\X$ has a decomposition \eqref{dec} where the processes $\bf M$ and $\bf A$ have the stated properties. We will start by considering the case where $\Lambda$ does not have a Gaussian component, i.e.\ $\sigma^2=0$.

\emph{Case 1. $\Lambda$ has no Gaussian component}:    We  divide the proof  into the following six steps. 

\emph{Step~1}: 
Let $X^0_t = X_t - \beta(t)$, with 
\begin{equation}
 \beta(t) = \int_{U} B\big(\phi(t, u), u\big) \, \kappa(d u). 
\end{equation}
We will give the  series representation for $\X^0$ that will be crucial for our considerations. To this end, define for $s\neq 0$ and $u \in U=\R\times V$ 
\begin{equation}\label{}
R(s, u) = \begin{cases}
\inf\{ x>0: \rho_u(x,\infty) \le s\} \qquad\qquad   & \text{if } s>0, \\ 
\sup\{ x<0: \rho_u(-\infty, x) \le -s\}  & \text{if } s<0. 
\end{cases}
\end{equation}
Choose a probability measure $\tilde \kappa$ on $U$ equivalent to $\kappa$, and let $h(u)= \frac{1}{2}(d \tilde \kappa/d\kappa)(u)$.
By an extension of  our probability space if necessary, \citet{Rosinski_series_point}, Proposition~2 and Theorem~4.1, shows that there exists three independent sequences
$(\Gamma_i)_{i\in \N}$,  $(\epsilon_i)_{i\in \N}$,   and $(T_i)_{i\in \N}$, where $\Gamma_i$ are partial sums of i.i.d.\ standard exponential random variables, $\epsilon_i$ are i.i.d.\ symmetric Bernoulli random variables, and $T_i=(T_i^1, T_i^2)$ are i.i.d.\ random variables in $U$ with the common distribution $\tilde \kappa$, such that for every $A\in \s$, 
\begin{equation}\label{rep_lambda}
 \Lambda(A)= \nu_0(A)+  \sum_{j=1}^\infty   [R_j\1_A(T_j)-\nu_j(A)\Big] \qquad \text{a.s.}
\end{equation}
where $R_j=R(\epsilon_j\Gamma_j h(T_j),T_j)$, \ $\nu_0(A)= \int_A b(u) \, \kappa(d u)$,  and for $j\ge 1$
\begin{equation}\label{}
\nu_j(A) = \int_{\Gamma_{j-1}}^{\Gamma_j} \E \lt R(\epsilon_1 r h(T_1),T_1)\rt \1_A(T_1) \, dr.
\end{equation}
It follows by the same argument that 
\begin{equation} \label{sY0}
X^0_t = \sum_{j=1}^{\infty} \big[ R_j \phi(t, T_j) - \alpha_j(t) \big] \qquad \text{a.s.},
\end{equation}
where 
\begin{equation} \label{}
\alpha_j(t) = \int_{\Gamma_{j-1}}^{\Gamma_j} \E \lt R(\epsilon_1 r h(T_1),T_1) \phi(t, T_1)\rt  \, dr.
\end{equation}

\emph{Step~2}: We will show that  for every $i \in \N$
\begin{equation} \label{eq:T2}
\Delta X_{T_i^1} = R_i \phi(T_i^1, T_i) \quad \text{a.s.}
\end{equation}
 Since $\X$ is \ca,  the series 
 \begin{equation} \label{eq:Yep}
X_t^0=\sum_{j=1}^{\infty} \big[ R_i \phi(t, T_i) - \alpha_i(t) \big] 
\end{equation} 
converges uniformly for $t$ in compact intervals a.s., cf.\ \citet[Corollary~3.2]{Ito-Nisio-D}. 
Moreover, $\beta$ is \ca, see \cite[Lemma~3.5]{Ito-Nisio-D},  and by Lebesgue's dominated convergence theorem it follows that 
$\alpha_j$, for $j\in \N$, are \ca\ as well. Therefore, with probability one,
\begin{equation} \label{}
\Delta X_t = \Delta \beta(t) + \sum_{j=1}^{\infty} \big[ R_j \Delta \phi(t, T_j) - \Delta \alpha_j(t) \big] \quad \text{for all} \ t>0.
\end{equation}
Hence, for every $i \in \N$ almost surely
\begin{equation} \label{eq:T1}
\Delta X_{T_i^1} = \Delta \beta(T_i^1) + \sum_{j=1}^{\infty} \big[ R_j \Delta \phi(T_i^1, T_j) - \Delta \alpha_j(T_i^1) \big] = 
\sum_{j=1}^{\infty} R_j \Delta \phi(T_i^1, T_j) 
\end{equation}
Indeed, by assumption \eqref{Lambda2} the distribution of $T_i^1$ is continuous. Since $\beta$ may  have only a countable number of discontinuities, with probability one $T_i^1$ is a continuity point of $\beta$. Hence $\Delta \beta(T_i^1)=0$ a.s. Since $(\Gamma_j)_{j\in \N}$ are independent of $T_i^1$, the argument used for $\beta$ also yields $\Delta \alpha_j(T_i^1)=0$ a.s. 
This proves \eqref{eq:T1}.  

Furthermore, for $i\ne j$ we get
\begin{align} \label{}
\P( \Delta \phi(T_i^1, T_j) \ne 0) &= \int_{U} \P( \Delta \phi(T_i^1, T_j) \ne 0\, |\, T_j= u) \, \tilde \kappa(d u) \\
&= \int_{U} \P( \Delta \phi(T_i^1, u) \ne 0 ) \, \tilde \kappa(d u) = 0
\end{align}
again because $\phi(\cdot, u)$ may have only countably many jumps and the distribution of $T_i^1$ is continuous. 
If $j=i$ then
\begin{equation} \label{}
\Delta \phi(T_i^1, T_i) = \lim_{h \downarrow 0,\, h>0} \big[\phi(T_i^1, T_i^1, T_i^2) - \phi(T_i^1-h, T_i^1, T_i^2) \big] = \phi(T_i^1, T_i) 
\end{equation}
as $\phi(t,s,v)=0$ whenever $t<s$. This simplifies \eqref{eq:T1} to \eqref{eq:T2}.

\emph{Step~3}: Next we will show that 
$\M$, defined  in \eqref{M}, is a well-defined \ca\ process satisfying
\begin{equation}\label{eq-MY}
 \Delta M_{T^1_i} =\Delta X_{T^1_i}\quad \text{a.s.\ for all }i\in \N.
\end{equation} 
Since any semimartingale has finite quadratic variation, we get with probability one
\begin{align*}
\infty & > \sum_{0<s\le t} \big(\Delta X_s\big)^2 \ge \sum_{0<T^1_i\le t} \big(\Delta X_{T_i^1}\big)^2 = 
\sum_{i = 1}^{\infty} \big[R_i \phi(T_i^1, T_i)\big]^2 \1_{\{0<T_i^1 \le t\}},
\end{align*}
where the last equation employs \eqref{eq:T2}. Put for $t,r \ge 0$ and $(\epsilon,s,v) \in \{-1,1\} \times \R \times V$
\begin{equation} \label{}
H(t; r, (\epsilon,s,v)) = R(\epsilon r h(s,v), (s,v)) \phi(s,s,v) \1_{\{0< s \le t\}}.
\end{equation}
The above bound shows that for each $t\ge 0$
\begin{equation} \label{}
\sum_{i=1}^{\infty}  |H(t; \Gamma_i, (\epsilon_i,T_i^1,T_i^2)) |^2 < \infty \quad \text{a.s.}
\end{equation}
That implies, by  \citet[Theorem 4.1]{Rosinski_series_point}, that  the following limit is finite
\begin{align*}
\lim_{n\to \infty} \int_0^n \E \lt H(t; r, (\epsilon_1,T_1^1,T_1^2))^2 \rt \, dr = \int_0^{\infty} \E \lt H(t; r, (\epsilon_1,T_1^1,T_1^2))^2 \rt \, dr.
\end{align*}
Evaluating this limit we get
\begin{align*}
\infty &> \int_0^{\infty} \E \lt R(\epsilon_1 r h(T_1), T_1) \phi(T_i^1, T_i) \1_{\{0<T_i^1 \le t\}}\rt^2 \, dr \\
& = \int_0^{\infty} \int_{\R\times V} \E\lt R(\epsilon_1 r h(s,v), (s,v)) \phi(s, s,v) \1_{\{0< s \le t\}}\rt^2 \, \tilde\kappa(ds,dv) dr \\
& = 2 \int_0^{\infty} \int_{\R\times V} \E\lt R(\epsilon_1 u, (s,v)) \phi(s, s,v) \1_{\{0< s \le t\}}\rt^2 \, \kappa(ds,dv) du \\
& = \int_{\R\times V} \int_{\R} \lt x \phi(s, s,v) \1_{\{0< s \le t\}}\rt^2 \, \rho_{(s,v)}(dx) \kappa(ds,dv) \\
& = \int_{(0,t] \times V} \int_{\R} \min\{ |x \phi(s, s,v)|^2, 1\} \, \rho_{(s,v)}(dx) \kappa(ds,dv). 
\end{align*}
Finiteness of this integral in conjunction with \eqref{drift} yield the existence of the stochastic integral
\begin{equation} \label{}
M_t = \int_{(0,t] \times V} \phi(s, s,v) \, \Lambda(ds,dv). 
\end{equation}
The fact that $\M$ has independent increments is obvious since $\Lambda$ is independently scattered, and its continuity in probability is a consequence of \eqref{Lambda2}.  We may choose a \ca\ modification of $\M$, cf.\ \cite[Theorem~11.5]{Sato}. Due to the independent increments,  $\M$ is a semimartingale if and only if its shift component $(\zeta_t)_{t\geq 0}$ is of finite variation, cf.\ \cite[Ch.~II, Corollary~5.11]{Jacod_S}, which follows from  \eqref{drift} and the fact that 
\begin{equation}\label{def-a}
 \zeta_t=\int_{(0,t]\times V} B\big(\phi(s,s,v),(s,v)\big)\,\kappa(ds,dv),\qquad t\geq 0,
\end{equation}
see \cite[Theorem~2.7]{Rosinski_spec}.
For $t\geq 0$ we can  write $M_t$ as a series using the series representation \eqref{rep_lambda} of $\Lambda$. It follows that
\begin{equation} \label{se-M}
M_t =\zeta_t+ \sum_{i=1}^{\infty} \big[ R_i \phi(T_i^1, T_i) \1_{\{0<T_i^1 \le t\}} - \gamma_j(t)\big]
\end{equation}
where 
\begin{equation}\label{}
 \gamma_j(t)=\int_{\Gamma_{j-1}}^{\Gamma_j} \E \lt R(\epsilon_1 r h(T_1),T_1) \phi(T^1_1,T_1)\1_{\{ 0<T^1_j\leq t\}})\rt  \, dr. 
\end{equation}
The assumption \eqref{Lambda2} yields that  $\zeta$ and $\gamma_j$ are continuous. Moreover, by arguments as above we have $\Delta M_{T^1_i}= R_i\phi(T^1_i,T_i)$ a.s.\ and hence by \eqref{eq:T2} we obtain \eqref{eq-MY}. 

\emph{Step~4}:  In the following we will show the existence of  a sequence $(\tau_k)_{k\in \N}$ of totally inaccessible stopping times such that  all local martingales $\mathbf{Z}=(Z_t)_{t\geq 0}$  with respect to $\Fi^\Lambda$ are purely discontinuous and with probability one 
\begin{equation}\label{con-stop}
\{t\geq 0: \Delta Z_t\neq 0\}\subseteq \{\tau_k:k\in \N\}\subseteq \{T^1_k:k\in \N\}.
\end{equation}
To show the above choose a sequence $(B_k)_{k\geq 1}$ of disjoint sets which generates $\mathcal V_0$ and for all  $k\in \N$ let ${\bf U}^k=(U^k_t)_{t\geq 0}$ be given by 
\begin{equation}\label{}
  U^k_t=\Lambda((0,t]\times B_k).
\end{equation}
 Assumption \eqref{Lambda2} implies that ${\bf U}^k$ is a continuous in probability process with independent increments and has therefore a \ca\ modification.  Hence   ${\bf U}=\{(U_t^k)_{k\in\N}:t\in \R_+\}$ is a  continuous in probability \ca\ $\R^\N$-valued process with no Gaussian component. Let $E=\R^\N\setminus\{0\}$. Then $E$ is a Blackwell space and  $\mu$ defined  by 
   \begin{equation}\label{}
 \mu(A)=\sharp\big\{t\in\R_+:(t,\Delta U_t)\in A\big\},\qquad A\in \mathscr{B}(\R_+\times E)
\end{equation} 
 is a Poisson random measure on $\R_+\times E$. Let $\nu$ be the  intensity measure of $\mu$. By assumption \eqref{Lambda2} we have that $\nu(\{t\}\times E)=0$ for all $t\geq 0$, moreover, $\Fi^\Lambda$ is the least filtration for which $\mu$ is an optional random measure. Thus according to   \cite{Jacod_S}, Ch.~III, Theorem~1.14(b) and the remark after Ch.~III, 4.35,  $\mu$ has the martingale representation property, that is for all real-valued local martingales ${\bf Z}=(Z_t)_{t\geq 0}$ with respect to $\Fi^\Lambda$ there  exists  a predictable function  $\phi$ from $\Omega\times \R_+\times E$ into $\R$ such that 
 \begin{equation}\label{mar-rep-N}
  Z_t = \phi*(\mu-\nu)_t,\quad t\geq 0 
 \end{equation} 
 (in \eqref{mar-rep-N} the symbol $*$ denotes integration with respect to $\mu-\nu$ as in \cite{Jacod_S}). 
 Thus by definition, see \cite[Ch.~II, Definition~1.27(b)]{Jacod_S},   $\bf Z$ is a purely discontinuous local martingale and $(\Delta Z_t)_{t\geq 0}=(\phi(t,\Delta U_t)\1_{\{\Delta U_t\neq 0\}})_{t\geq 0}$, which shows that with probability one
 \begin{equation}\label{}
 \{t\geq 0: \Delta Z_t\neq 0\}\subseteq \{t\geq 0: \Delta U_t\neq 0\}.
 \end{equation}

 All real-valued continuous in probability \ca\  processes ${\bf Y}=(Y_t)_{t\geq 0}$  with independent increments  is quasi-left continuous cf.\ \cite[Ch.~II, Corollary~5.12]{Jacod_S}, and hence there exists a sequence of totally inaccessible stopping times  that exhausts  the jumps of $\bf Y$, cf.\ \cite[Ch.~I, Proposition~2.2]{Jacod_S}.  
  Thus by a diagonal argument we may exhausts the jumps of $\bf U$ by a sequence of totally inaccessible stopping times $(\tau_k)_{k\in \N}$, that is  
    \begin{equation}\label{ex_sw}
 \{\tau_k:k\in \N\}= \{t\geq 0: \Delta U_t\neq 0\}.
\end{equation}  
Arguing as in Step~2  with $\phi(t,s,v)=\1_{(0, t]}(s)\1_{ B_k}(v)$ shows that  with probability one
\begin{equation}\label{jumps-U-12}
  \Delta U^k_t
 = \Delta \zeta(t) + \sum_{j=1}^{\infty} \big[ R_j  \1_{\{t=T^1_j\}} \1_{\{T^2_j\in B_k\}}- \Delta \gamma_j(t) \big] \quad \text{for all} \ t>0
\end{equation} 
where 
\begin{align}\label{}
 \xi(t)={}&  \int_{\R\times V} \1_{\{0\leq s\leq t\}}\1_{\{v\in B_k\}}b(s,v)\, \kappa(ds,dv),\\ 
 \gamma_j(t)={}&  \int_{\Gamma_{j-1}}^{\Gamma_j} \E \lt R(\epsilon_1 r h(T_1),T_1) \1_{\{ T^1_j\leq t\}}\1_{\{T^2_j\in B_k\}})\rt  \, dr.
\end{align}
By assumption \eqref{Lambda2}, $\xi$ and $\gamma_j$ are continuous and hence with probability one
\begin{equation}\label{}
 \Delta U^k_t= \sum_{j=1}^{\infty} R_j  \1_{\{t=T^1_j\}} \1_{\{T^2_j\in B_k\}} \qquad \text{for all }t>0.
\end{equation}
which shows that 
\begin{equation}\label{sub-T}
  \{\tau_k: k\in \N\}\subseteq \{T^1_k: k\in \N\}\qquad \text{a.s.} 
\end{equation}
and completes the proof of Step~4. 

\emph{Step~5}:
Fix $r\in \N$ and let  $\X'=(X'_t)_{t\geq 0}$ be given by 
\begin{equation}\label{}
 X' _t=X_t-\sum_{s\in (0,t]} \Delta X_s\1_{\{\abs{\Delta X_s}>r\}}.
\end{equation}
We will show that $\X'$ is a special semimartingale with martingale component $\M'=(M'_t)_{t\geq 0}$ given by 
\begin{equation}\label{def-M'}
M_t'=\tilde M_t-\E \tilde M_t \quad \text{where}\quad \tilde M_t=M_t-\sum_{s\in (0,t]}\Delta M_s \1_{\{\abs{M_s}>r\}}.
\end{equation}
 Recall that $\bf M$ is given by \eqref{M}.
To show above  we note that    $\X'$ is a special semimartingale since its jumps are bounded by $r$ in absolute value; denote by $\bf W$ and $\bf N$ the finite variation and martingale compnents, respectively,  in the canonical decomposition $\X'=X_0+\mathbf{W}+\mathbf{N}$ of $\X'$. That is, we want to show  that $\mathbf{N}= \M'$. Process $\M'$, given by \eqref{def-M'}, is obviously a martingale and by \eqref{eq-MY} we have   for all $i\in \N$ 
  \begin{equation}\label{eq-M'}
    \Delta M'_{T_i^1}=\Delta M_{T_i^1}\1_{\{\abs{\Delta M_{T_i^1}}\leq r\}}=\Delta X_{T_i^1}\1_{\{\abs{\Delta X_{T_i^1}}\leq r\}}=\Delta X'_{T_i^1} \qquad \text{a.s.}
  \end{equation}
 Since $\textbf W$ is predictable and $\tau_k$ is a totally inaccessible stopping time we have that $\Delta W_{\tau_k}=0$ a.s.\ cf.\ \cite[Ch.~I, Proposition~2.24]{Jacod_S} and hence 
   \begin{equation}\label{eq_N_j}
 \Delta N_{\tau_k}=\Delta X_{\tau_k}'-\Delta W_{\tau_k}=\Delta X_{\tau_k}'=\Delta M_{\tau_k}'\qquad \text{a.s.}
 \end{equation}
 the last equality follows by \eqref{eq-M'} and the second inclusion in \eqref{con-stop}. Since  $\bf N$ and $\M'$ are local martingales they are in fact  purely discontinuous local martingale, cf.\ Step~3,  and   with probability one 
\begin{equation}\label{}
  \{t\geq 0:\Delta N_t\neq 0\}\subseteq \{\tau_k: k\in \N\},\quad  \{t\geq 0:\Delta M'_t\neq 0\} \subseteq \{\tau_k: k\in \N\}.
\end{equation}
According to   \eqref{eq_N_j} this shows that $(\Delta N_t)_{t\geq 0}=(\Delta M'_t)_{t\geq 0}$, and we conclude that ${\bf N}=\M'$ since  $\bf N$ and $\M'$ are purely discontinuous local martingales, see \cite[Ch.~I, Corollary~4.19]{Jacod_S}. 
This completes Step~4. 

\emph{Step~6}: We will show that   $\A$, given by \eqref{A}, is a predictable \ca\ process of finite variation. By linearity,   $\A$  is a well-defined \ca\ process.  According to Step~5 the process  $\textbf{W}:=\X'-X_0-\M'$ is predictable and has \ca\ paths of finite variation.
Thus with $\mathbf{V}=(V_t)_{t\geq 0}$  given by 
  \begin{equation}
V_t= \sum_{s\in (0,t]} \Delta X_s\1_{\{\abs{X_s}>r\}}-\sum_{s\in (0,t]} \Delta M_s\1_{\{\abs{M_s}>r\}}
  \end{equation}
we have by the definitions of $\bf W$ and $\bf V$   that 
\begin{equation}\label{decomp-A-67}
  A_t= X_t-X_0-M_t=W_t+V_t-\E \tilde M_t.
  \end{equation}
  %
This shows that $\A$ has   \ca\ sample paths of finite variation. Next we will show that $\A$ is predictable. Since the processes $\bf W$, $\bf V$ and $\tilde \M$ depend on the truncation level $r$ they will be denoted  $\mathbf{W}^r$, $\mathbf{V}^{r}$ and $\tilde \M^r$ in the following.  As $r\to \infty$, $V_t^r(\omega)\to 0$ point wise in $(\omega,t)$, which by \eqref{decomp-A-67} shows that $W^r_t(\omega)-\E \tilde M^r_t \to A_t(\omega)$  point wise in $(\omega,t)$ as $r\to \infty$. For all $r\in \N$,    $(W^r_t-\E \tilde M^r_t)_{t\geq 0}$ is a predictable process, which implies  that  $\bf A$ is a point wise limit of predictable processes and hence predictable. This completes the proof of Step~6.

%
\emph{Case 2. $\Lambda$ is Gaussian}:
Assume that $\Lambda$ is a symmetric  Gaussian random measure. By \citet[Theorem~4.6]{Andreas2} used on $C_t=(-\infty,t]\times V$,  $\X$ is a special semimartingale in $\Fi^{\Lambda}$ with martingale component $\M=(M_t)_{t\geq 0}$ given by 
   \begin{equation}\label{}
 M_t=\int_{(0,t]\times V} \phi(s,s,v)\,\Lambda(ds,dv),\qquad t\geq 0,
\end{equation}
see \cite[Equation~(4.11)]{Andreas2}, which completes the proof in the   Gaussian case. 

\emph{Case 3. $\Lambda$ is general}:
Let us observe that it is enough to show the theorem in the above two cases. We may decompose $\Lambda$ as $\Lambda = \Lambda_G + \Lambda_P$, where $\Lambda_G, \Lambda_P$ are independent, independently scattered random measures. $\Lambda_G$ is a symmetric Gaussian random measure characterized by \eqref{Lambda} with $b \equiv 0$ and $\kappa \equiv 0$ while $\Lambda_P$ is given by \eqref{Lambda} with $\sigma^2 \equiv 0$. Observe that 
\begin{equation} \label{eq-fil}
\Fi^{\Lambda} = \Fi^{\Lambda_G} \vee \Fi^{\Lambda_P},
\end{equation}
which can be deduced  from the L\'evy-It\^o decomposition used processes $\Y=(Y_t)_{t\geq 0}$ of the form $Y_t=\Lambda((0,t]\times B)$ where  $B\in \V_0$.  We  have $\X = \X^G + \X^P$, where $\X^G$ and $\X^P$ are defined by \eqref{def-X} with $\Lambda_G$ and $\Lambda_P$ in the place of $\Lambda$, respectively. Since $(\Lambda, \X)$ and $(\Lambda_P -\Lambda_G, \X^P- \X^G)$ have the same distributions,
the process $\X^P -\X^G$ is a semimartingale with respect to $\Fi^{\Lambda_P-\Lambda_G}= \Fi^{\Lambda_P} \vee \Fi^{-\Lambda_G}= \Fi^{\Lambda}$. 
Consequently, processes $\X^G$ and $\X^P$ are semimartingales with respect to $\Fi^{\Lambda}$, and so, they are semimartingales relative to $\Fi^{\Lambda_G}$ and $\Fi^{\Lambda_P}$, respectively, and the general result follows from the above two  cases. 

\emph{The uniqueness part}: Assume that $\X$ has a representation of the form   $\X=X_0+\M'+\A'$ where $\M'$ is a continuous in probability semimartingale with independent increments and $\A'$ is a \ca\ predictable process of finite variation.
With  $\M$ and $\A$  given by \eqref{M} and \eqref{A} we need to show that  process  $\mathbf{Y}$, defined by 
\begin{equation}\label{rep-Y-u}
\mathbf{Y}=\M-\M'=\A'-\A,
\end{equation} 
is a continuous deterministic function of finite variation. The first equality in \eqref{rep-Y-u} shows that $\bf Y$ is  quasi-left continuous, cf.\ \cite[Ch.~II, Corollary~4.18]{Jacod_S},  and  the second equality shows that $\bf Y$ is predictable, which together implies that $\bf Y$ is continuous, use \cite[Ch.~I, Proposition~2.24+Definition~2.25]{Jacod_S}. By the L\'evy-It\^o decompositions of $\M$ and $\M'$ it follows that  $\Y=\mathbf{U}+f$ where $\bf U$ is a continuous martingale and $f$ is a continuous deterministic function of finite variation; that $f$ is of finite variation follows by \cite[Ch.~II, Corollary~5.11]{Jacod_S}. Since $\bf Y$ is of finite variation we deduce that $\mathbf{U}=0$, that is,  $\mathbf{Y}=f$.  This  completes the proof of the uniqueness.  

\emph{The special semimartingale part}:
To prove the  part concerning the special semimartingale property of $\X$ we note that the process $\bf A$ in \eqref{A} is a special semimartingale since it is predictable and of finite variation. Thus $\bf X$ is a special semimartingale if and only if $\bf M$ is special semimartingale. Due to its independent increments,  $\bf M$ is a special semimartingale if and only if $\E\abs{M_t}<\infty$ for all $t>0$, cf.\  \cite[Ch.~II, Proposition~2.29(a)]{Jacod_S}, and in that case $M_t=(M_t-\E M_t)+\E M_t$ is the  canonical decomposition of $\bf M$. This completes the proof.
\end{proof}

\medskip

\begin{remark}\label{}
{ \rm
We conclude this section by recalling that the proof of any of the results on Gaussian semimartingales $\X$ mentioned in the Introduction relies on the approximations of the finite variation component $\A$ by discrete time Doob--Meyer decompositions $\A^n=(A^n_t)_{t\geq 0}$ given by 
\begin{equation}
A^n_t=\sum_{i=1}^{[2^n t]} \E[X_{i2^{-n}}-X_{(i-1)2^{-n}}|\F_{(i-1)2^{-n}}],
\qquad t\geq 0
\end{equation}
and showing that the convergence $\lim_n A^n_t=A_t$ holds in an appropriate sense, see \cite{Meyer_app}.
This technique does not  seem  effective in the non-Gaussian situation since it relies on strong integrability properties of  functionals of $\X$.  
}
\end{remark}

\section{Some stationary increment semimartingales} \label{appl}


In this section we consider infinitely divisible  processes which are stationary increment mixed moving averages (SIMMA).
Specifically, a process $\X=\p X$ is called a SIMMA process if it can be written in the form
\begin{equation} \label{eq:mma}
X_t=\int_{\R \times V} \big[f(t-s, v)- f_0(-s, v)\big]\, \Lambda(ds, dv),\qquad t\geq 0,
\end{equation} 
where the functions $f$ and $f_0$ are deterministic measurable such that  $f(s, v) = f_0(s,v) = 0$ whenever $s<0$, and $f(\cdot,v)$ is \ca\ for all $v$.  $\Lambda$ is an independently scattered infinitely divisible random measure that is invariant under translations over $\R$. If $V$ is a one-point space (or simply, there is no $v$-component in \eqref{eq:mma}) and $f_0=0$, then \eqref{eq:mma} defines a moving average (a mixed moving average for a general $V$, cf. \cite{mixed_ma_R}). If $V$ is a one-point space  
and $f_0(x)=f(x)=x_+^\alpha$ for some $\alpha\in \R$,  then $\X$ is  a fractional L\'evy process.

The finite variation property of SIMMA processes was investigated in \citet{fv-mma} and these results, together with Theorem \ref{thm1}, are crucial in our description of SIMMA semimartingales.

The random measure $\Lambda$ in \eqref{eq:mma} is as in \eqref{Lambda} but the functions $b$ and $\sigma^2$ do not depend on $s$ and the measure $\kappa$ is a product measure: $\kappa(ds,dv)=ds\,m(dv)$ for some $\sigma$-finite measure $m$ on $V$.  In this case, for $A\in \s$ and $\theta\in \R$  
\begin{align}\label{eq:M}
&\log \mathbb{E} e^{i\theta \Lambda(A)} \\
& \quad = \int_A \Big( i\theta b(v)-\frac{1}{2}\theta^2 \sigma^2(v) +\int_{\R} (e^{i\theta x}-1-iu\lt x\rt) \, \rho_v(dx)\Big)\, ds\,m(dv).
\nonumber
\end{align}
The function $B$ in \eqref{B} is independent of $s$, so that with $B(x,v)=B(x,(s,v))$ we have
\begin{equation}\label{B1}
   B(x, v) = xb(v) +  \int_{\R} \big( \lt xy\rt - x \lt y\rt\big) \, \rho_{v}(dy), \quad x \in \R, \ v \in V.
\end{equation}
Notice that $\Lambda$ satisfies \eqref{Lambda2} since  $\kappa(ds,dv)=ds\,m(dv)$. 

The SIMMA process \eqref{eq:mma} is a special case of \eqref{def-X} if we take $\phi(t,s,v)=f(t-s, v)- f_0(-s, v)$. Therefore, from Theorem \ref{thm1} we obtain:

\begin{theorem}\label{decomp-special}
Let $\mathbf{X}=\p X$ be of the form 
\begin{equation}\label{def_Y_phi-levy} 
X_t=\int_{\R \times V} \big[f(t-s, v)- f_0(-s, v)\big]\, \Lambda(ds, dv),\qquad t\geq 0,
\end{equation}
specified by \eqref{eq:mma}--\eqref{eq:M}, and let $B$ be given by \eqref{B1}. Assume that  
\begin{equation} \label{drift-4}
\int_{V} \big| B\big(f(0, v), v\big) \big| \, m(dv) < \infty.
\end{equation}
Then  $\mathbf{X}$ is a semimartingale with respect to the filtration 
$\mathbb{F}^{\Lambda}=(\mathcal{F}^{\Lambda}_t)_{t \ge 0}$ if and only if
\begin{equation}\label{dec-levy}
   X_t = X_0 + M_t + A_t, \quad t \ge 0,
\end{equation}
where $\M= \p M$ is a  L\'evy process given by  
\begin{equation} \label{M-levy}
M_t = \int_{(0,t] \times V} f(0,v)\,\Lambda(ds,dv), \quad t\ge 0, 
\end{equation}
(i.e., the integral exists), 
and $\mathbf{A}= \p A$ is a predictable process of finite variation  given by 
\begin{equation} \label{A-levy}
A_t = \int_{\R \times V} [g(t-s,v) -  g(-s,v)] \,\Lambda(ds,dv)
\end{equation}
where $g(s, v)= f(s, v) - f(0, v) \1_{\{s\geq 0\}}$. 
%
\end{theorem}

Now we will give specific and closely related necessary and sufficient 
conditions on $f$ and $\Lambda$ that make $\X$ a semimartingale.

 \begin{theorem}[Sufficiency]\label{thm-suf} 
Let $\mathbf{X}=\p X$ be specified by \eqref{eq:mma}--\eqref{eq:M}. Suppose that 
\eqref{drift-4} is satisfied 
and that for $m$-a.e.\ $v \in V$,  $f(\cdot,v)$ is absolutely  continuous on $[0,\infty)$ with a derivative 
$\dot  f(s,v)=\frac{\partial }{\partial s}f(s,v)$ satisfying
  \begin{align}
& \int_V \int_{0}^\infty\big(\abs{\dot f(s,v)}^2\sigma^2(v)\big)\,ds
 \,m(dv)<\infty,   \label{int_1} \\
 \label{Cf}
  &  \int_V \int_0^\infty \int_\R \big(\abs{x\dot f(s,v)}\wedge
   \abs{x\dot f(s,v)}^2\big)\,\rho_v(dx)\,ds\,m(dv)<\infty.
\end{align}
Then $\bf X$ is a semimartingale with respect to $\Fi^\Lambda$.  
\end{theorem}

 \begin{proof}
 We need to verify the conditions of Theorem \ref{decomp-special}.
 We see that for $m$-a.e.\ $v \in V$, $g(\cdot,v)$ is absolutely continuous on $\R$ with derivative $\dot g(s,v)=\dot f(s,v)$ for $s > 0$ and $\dot g(s,v)=0$ for $s<0$. By  Jensen's inequality,  for each fixed $t>0$, the function
 \begin{equation}
(s,v)\mapsto g(t-s,v)-g(-s,v)=\int_0^t \dot g(u-s,v)\,du,
 \end{equation}
when substituted for $\dot  f(s,v)$ in \eqref{int_1}--\eqref{Cf}, satisfies these conditions. Indeed,  it is  straightforward to verify \eqref{int_1}. To verify  \eqref{Cf} we use the fact that $\psi\! :u\mapsto 2\int_0^{\abs{u}} (v\wedge 1)\,dv$ is convex and satisfies $\psi(u)\leq \abs{ux}\wedge \abs{ux}^2\leq 2\psi(u)$. In particular,  $(s,v)\mapsto g(t-s,v)-g(-s,v)$ satisfies  \eqref{i2} of the Introduction, and so does the function
 \begin{equation}
 (s,v)\mapsto f(0,v) \1_{(0, t]}(s)=g(t-s,v) -  g(-s,v) - [f(t-s,v) -  f(-s,v)].  
 \end{equation}
This fact together with assumption  \eqref{drift-4} guarantee that $\M$ of Theorem~\ref{decomp-special} is well-defined. Then $\A$ is well-defined by \eqref{dec-levy}. The process $\A$ is of finite variation by \cite[Theorem~3.1]{fv-mma} because $g(\cdot,v)$ is absolutely continuous on $\R$ and $\dot g(\cdot, v)= \dot f(\cdot, v)$ satisfies \eqref{int_1}--\eqref{Cf}.
  \end{proof}
%

 \begin{theorem}[Necessity]\label{thm-nes}
 Suppose   that $\X$ is a semimartingale with respect to $\Fi^\Lambda$ and 
 for $m$-almost every $v\in V$ we have either 
   \begin{equation}\label{invar-con}
\int_{-1}^1 \abs{x}\,\rho_{v}(dx)=\infty\quad \text{or}\quad 
  \sigma^2(v)>0.
 \end{equation}
 Then for $m$-a.e.\ $v$,     $ f(\cdot,v)$ is absolutely continuous on
 $[0,\infty)$ with a derivative $\dot f(\cdot,v)$ satisfying
 \eqref{int_1} and  
  \begin{align}
  \label{trunc_case} & \int_0^\infty\int_\R \big(\abs{x \dot
    f(s,v)}\wedge \abs{x \dot f(s,v)}^2\big)(1 \wedge
  x^{-2})\,\rho_v(dx)\, ds<\infty.   
  \end{align} 
  If, additionally,
\begin{equation}\label{eq:u0}
\limsup_{u \to \infty} \, \frac{u\int_{\abs{x}>u}
  \abs{x}\,\rho_v(dx)}{\int_{|x|\le u} x^2 \, \rho_v(dx)} < \infty
\qquad m\text{-a.e.} 
\end{equation}
then for $m$-a.e.\ $v$,
\begin{equation}\label{fdot_int}
 \int_{0}^\infty \int_{\R} ( |x{\dot f}(s,v)|^2 \wedge |x{\dot
   f}(s,v)|) \, \rho_v(dx)\, ds < \infty. 
 \end{equation}
Finally, if 
 \begin{equation}\label{eq:u00}
\sup_{v\in V}\sup_{u > 0} \, \frac{u\int_{\abs{x}>u}
  \abs{x}\,\rho_v(dx)}{\int_{|x|\le u} x^2 \, \rho_v(dx)} <\infty 
\end{equation} 
then  $\dot f$ satisfies \eqref{int_1}--\eqref{Cf}. 
  \end{theorem}
  
   \begin{proof}
 Assume that $\X$ is a semimartingale with respect to $\Fi^{\Lambda}$. 
 By a symmetrization argument we may assume that $\Lambda$ is a symmetric random measure. Indeed,  let $\Lambda'$ be an independent copy of $\Lambda$ and $\X'$ be defined by \eqref{eq:mma} with $\Lambda$ replaced by $\Lambda'$. Then $\X'$ is a semimartingale with respect to  $\Fi^{\Lambda'}$. By the independence, both $\X$ and $\X'$ are semimartingales with respect to  $\Fi^\Lambda\vee \Fi^{\Lambda'}$ and since $\Fi^{\Lambda-\Lambda'}\subseteq \Fi^\Lambda\vee \Fi^{\Lambda'}$, the process $\X-\X'$ is a semimartingale with respect to  $\Fi^{\Lambda-\Lambda'}$. This shows that we may assume that $\Lambda$ is symmetric. Then  \eqref{drift-4} holds since $B=0$.   
 
 By Theorem~\ref{decomp-special}  process $\bf A$ in \eqref{A-levy} is of finite variation. 
It follows from \cite[Theorem~3.3]{fv-mma}  that for $m$-a.e.\ $v$, $g(\cdot,v)$ is absolutely continuous on $\R$ with a derivative $\dot g(\cdot,v)$ satisfying \eqref{int_1} and \eqref{trunc_case}.  Furthermore $\dot g$ satisfies \eqref{fdot_int} under  assumption \eqref{eq:u0},  and under assumption \eqref{eq:u00}, $\dot g$ satisfies   \eqref{Cf}. Since $f(s,v)=g(s,v)+f(0,v)\1_{\{s\geq 0\}}$, $f(\cdot,v)$ is absolutely continuous on $[0,\infty)$ with a derivative $\dot f(\cdot,v)=\dot g(\cdot,v)$ for $m$-a.e.\ $v$  satisfying the conditions of the theorem.  
 \end{proof}

  \begin{remark}\label{cor-iff}
Theorem~\ref{thm-nes} becomes an exact converse to Theorem~\ref{thm-suf} when \eqref{invar-con} holds and either  \eqref{eq:u0} holds and $V$ is a finite set, or \eqref{eq:u00} holds. 
\end{remark}

  \begin{remark}\label{remark_1}
Condition~\eqref{invar-con} is in general necessary to deduce that $f$
has absolutely continuous sections. Indeed, let $V$ be a one point
space so that $\Lambda$ is generated by increments of a L\'evy process
denoted again by $\Lambda$. If \eqref{invar-con} is not satisfied,
then taking  $f=\1_{[0,1]}$ we get that $X_t=\Lambda_t-\Lambda_{t-1}$
is of finite variation and hence a semimartingale, but $f$ is not
continuous on $[0,\infty)$. 
\end{remark}

\smallskip

Next we will consider several consequences of Theorems~\ref{thm-suf}--\ref{thm-nes}. When there is no $v$-component, \eqref{drift-4} is always satisfied and $\Lambda$ is generated by a two-sided  L\'evy process.
In what follows, $\mathbf{Z}=(Z_t)_{t\in \R}$ will denote a non-deterministic two-sided  L\'evy process, with characteristic triplet $(b,\sigma^2,\rho)$ and  $Z_0=0$.  $\Fi^Z$ will denote the least  filtration satisfying the usual conditions such that   $\sigma(Z_u:u\in (-\infty,t])\subseteq \F^Z_t$ for all $t\geq 0$.
\smallskip	

The following proposition characterizes fractional L\'evy processes which are semimartingales, and  completes results of \cite[Corollary~5.4]{Basse_Pedersen} 
and parts of \cite[Theorem~1]{Be_Li_Sc}.

\begin{proposition}[Fractional L\'evy processes]\label{fLp}
 Let $\gamma >0$,   $x_+:=\max\{x,0\}$ for $x\in \R$, $\mathbf{Z}$ be a L\'evy process as above, and   $\bf X$ be a fractional L\'evy process defined by  \begin{equation}\label{def-frac}
X_t=\int_{-\infty}^t \big\{(t-s)^\gamma_+-(-s)_+^\gamma\,\big\}\,dZ_s
\end{equation}
where   the stochastic integrals exist. 
Then  $\bf X$ is a semimartingale with respect to $\Fi^Z$ if and only if  $\sigma^2=0$, $\gamma\in (0,\tfrac{1}{2})$ and 
\begin{equation}\label{int-finite-786}
\int_\R \abs{x}^{\frac{1}{1-\gamma}}\,\rho(dx)<\infty. 
\end{equation}
\end{proposition}

\begin{proof}
First we notice that, as a consequence  of  $\X$ being   well-defined, $\gamma<\tfrac{1}{2}$ and 
\begin{equation}\label{equ-34}
\int_{\abs{x}>1} \abs{x}^{\frac{1}{1-\gamma}}\,\rho(dx)<\infty.
\end{equation}
Indeed, since the stochastic integral \eqref{def-frac} is well-defined,  \cite[Theorem~2.7]{Rosinski_spec} shows that 
\begin{equation}\label{Ros-Raj-23}
\int_{-\infty}^t \int_\R \big(1\wedge \abs{\{(t-s)^\gamma-(-s)^\gamma_+\}x}^2\big)\,\rho(dx)\,ds<\infty, \qquad t\geq 0.
\end{equation}
This implies that  $\gamma<\tfrac{1}{2}$ if $\rho(\R)>0$. A similar argument shows that $\gamma<\tfrac{1}{2}$ if $\sigma^2>0$, and thus, by the non-deterministic assumption on $\bf Z$,  we have shown  that $\gamma<\frac{1}{2}$.  Putting $t=1$ in~\eqref{Ros-Raj-23} and using  the estimate $\abs{(1-s)^\gamma-(-s)^\gamma_+}\geq \abs{\gamma (1-s)^{\gamma-1}}$ for $s\in (-\infty,0]$  we get
\begin{align}
\infty>{}&\int_{-\infty}^0 \int_\R \big(1\wedge \abs{\gamma(1-s)^{\gamma-1}x}^2\big)\,\rho(dx)\,ds\\
= {}& \int_\R \int_1^\infty \big(1\wedge \abs{\gamma s^{\gamma-1}x}^2\big)\,ds\,\rho(dx)
\\ \geq {}& \int_\R \int_{1\leq s\leq \abs{\gamma x}^{\frac{1}{1-\gamma}} }\,ds\,\rho(dx)\geq  \int_{\abs{\gamma x}>1} \big(\abs{\gamma x}^{\frac{1}{1-\gamma}} -1\big)\,\rho(dx).
\end{align}
This shows \eqref{equ-34}.

Suppose that $\bf X$ is a semimartingale. 
If  $\sigma^2>0$, then according to   Theorem~\ref{thm-nes},   $f$ is absolutely continuous on $[0,\infty)$ with a derivative $\dot f$ satisfying
\begin{align}
\int_0^\infty \abs{\dot f(t)}^2\,dt=\int_0^\infty \gamma^2 t^{2(\gamma-1)}\,dt<\infty\end{align}
which  is a contradiction and shows that $\sigma^2=0$. 
By the  non-deterministic assumption on  $\bf Z$ we have $\rho(\R)>0$.  To complete the proof of the necessity part, it remains to show that
\begin{equation}\label{equ-34a}
\int_{\abs{x}\le 1} \abs{x}^{\frac{1}{1-\gamma}}\,\rho(dx)<\infty.
\end{equation}
Since $\dot f(t)=\gamma t^{\gamma-1}$ for $t>0$, we have
\begin{equation}\label{basic-cal}
\int_0^\infty \big\{\abs{x \dot f(t)}\wedge \abs{x \dot f(t)}^2\big\}\,dt=C \abs{x}^{\frac{1}{1-\gamma}} 
\end{equation}
where $C=\gamma^{\frac{1}{1-\gamma}}(\gamma^{-1}+(1-2\gamma)^{-1})$. 
In the case  $\int_{\abs{x}\leq 1} \abs{x}\,\rho(dx)<\infty$ \eqref{equ-34a} holds    since $1<\tfrac{1}{1-\gamma}$. Thus we may assume that $\int_{\abs{x}\leq 1}\abs{x}\,\rho(dx)=\infty$, that is, \eqref{invar-con} of Theorem~\ref{thm-nes} is satisfied.  By 
Theorem~\ref{thm-nes}~\eqref{trunc_case} and \eqref{basic-cal} we have 
\begin{equation}
\int_{\abs{x}\leq 1} \abs{x}^{\frac{1}{1-\gamma}} \,\rho(dx)\leq \int_\R \abs{x}^{\frac{1}{1-\gamma}} (1\wedge x^{-2})\,\rho(dx)<\infty
\end{equation}
which completes the proof of the necessity part.   

On the other hand, suppose that $\sigma^2=0$, $\gamma\in (0,\tfrac{1}{2})$ and \eqref{int-finite-786} is satisfied. By \eqref{int-finite-786} and \eqref{basic-cal},  $f$ is absolutely continuous on $[0,\infty)$ with a derivative $\dot f$ satisfying \eqref{Cf} and hence $\bf X$ is a semimartingale with respect to $\Fi^Z$, cf.\ Theorem~\ref{thm-suf}. 
\end{proof}


\medskip

Below we will recall the conditions from  \cite{fv-mma} under which \eqref{eq:u0} or \eqref{eq:u00} hold.
Recall that a measure $\mu$ on $\R$ is said to be regularly varying if $x\mapsto \mu([-x,x]^c)$  is a regularly varying function; see \cite{Bingham}.  

\begin{proposition}[\citemma]\label{remark}  
Condition \eqref{eq:u0} is satisfied when one of the following two conditions holds for $m$-almost every $v \in V$
\begin{enumerate}[(i)]
	\item  \label{pro-con-1}  $\int_{\abs{x}>1} x^2 \, \rho_v(dx)<\infty$ or
\item  \label{pro-con-2} $\rho_v$ is regularly varying at $\infty$ with
          index $\beta\in [-2,-1)$. 
\end{enumerate}
Suppose that $\rho_v=\rho$ for all $v$, where $\rho$ satisfies \eqref{eq:u0} and is regularly varying with index $\beta\in (-2,-1)$ at 0. Then  \eqref{eq:u00} holds.  
\end{proposition}

Theorems~\ref{thm-suf}--\ref{thm-nes} and Proposition~\ref{remark} gives the following:

\begin{corollary}\label{cor-levy}
Suppose that   $\mathbf{Z}=(Z_t)_{t\in \R}$ is a two-sided  L\'evy process as above, with paths of  infinite variation on compact intervals. Let $\mathbf{X}=(X_t)_{t\geq 0}$ be a process of the form 
\begin{equation}\label{MA-case}
X_t=\int_{-\infty}^t \big\{f(t-s)-f_0(-s)\big\}\,dZ_s.
\end{equation}
Suppose that the random variable $Z_1$ is either square-integrable or has a regularly varying distribution at $\infty$ of index $\beta\in [-2,-1)$. 
Then $\bf X$ is a semimartingale with respect to $\Fi^Z$  if and only if $f$ is absolutely continuous on $[0,\infty)$ with a derivative $\dot f$ satisfying 
\begin{align}
& \int_0^\infty \abs{\dot f(t)}^2\,dt<\infty \qquad \text{if }\sigma^2>0, \\
& \label{levy-int-cor}\int_0^\infty \int_\R \big(\abs{x \dot f(t)}\wedge \abs{x \dot f(t)}^2\big)\,\rho(dx)\,dt<\infty.
\end{align}
\end{corollary}

\begin{proof}[Proof Corollary~\ref{cor-levy}]
The conditions imposed  on $Z_1$ are equivalent to that $\rho$ statisfies \eqref{pro-con-1} or \eqref{pro-con-2} of Proposition~\ref{remark}, respectively, cf.\  \cite[Theorem~1]{reg-var} and \cite[Theorem~25.3]{Sato}. Moreover, \eqref{invar-con} of Theorem~\ref{thm-nes} is equivalent to that $\bf Z$ has sample paths of infinite variation on compacs and hence the result follows by Theorems~\ref{thm-suf}--\ref{thm-nes}. 
\end{proof}


\begin{example}\label{ex-stable}
 In following we will consider  $\bf X$ and $\bf Z$  given as in Corollary~\ref{cor-levy} where  $\bf Z$ is either a stable or a tempered stable L\'evy process. 

\textbf{(i) Stable:}
Assume that  $\bf Z$ is a symmetric $\alpha$-stable L\'evy process with index $\alpha\in (1,2)$, that is, $\rho(dx)=c\abs{x}^{-\alpha-1}\,dx$ where $c>0$, and $\sigma^2=b=0$.  
Then $\bf X$ is a semimartingale with respect to $\Fi^Z$ if and only if $f$ is absolutely continuous on $[0,\infty)$ with a derivative $\dot f$ satisfying 
\begin{equation}\label{ex-stable-eq}
\int_0^\infty \abs{\dot f(t)}^\alpha\,dt<\infty.
\end{equation}
 We use  Corollary~\ref{cor-levy} to show the above. Note that  $\int_{\abs{x}\leq 1} \abs{x}\,\rho(dx)=\infty$ and $\rho$ is regularly varying at $\infty$ of index $-\alpha\in (-2,-1)$. Moreover, the identity
\begin{equation}\label{cal-stable}
\int_\R \big(\abs{xy}\wedge \abs{xy}^2\big)\,\rho(dx)=C \abs{y}^\alpha,\qquad y\in \R, 
\end{equation}
 with $C=2c ((2-\alpha)^{-1}+(\alpha-1)^{-1})$, shows that \eqref{levy-int-cor} is equivalent to \eqref{ex-stable-eq}. Thus the result follows by Corollary~\ref{cor-levy}.
 
 \textbf{(ii) Tempered stable:} Suppose  that $\bf Z$ is a symmetric tempered stable L\'evy process with indexs $\alpha\in [1,2)$ and $\lambda>0$, i.e., $\rho(dx)=c \abs{x}^{-\alpha-1}e^{-\lambda \abs{x}} \,dx$ where $c>0$,  and $\sigma^2=b=0$.  
Then $\bf X$ is a semimartingale with respect to $\Fi^Z$ if and only if $f$ is absolutely continuous on $[0,\infty)$ with a derivative $\dot f$ satisfying 
\begin{equation}\label{eq-temp-stable}
\int_0^\infty\big(\abs{ \dot f(t)}^{\alpha}\wedge \abs{ \dot f(t)}^2\big) \, ds <\infty.
\end{equation}
Again we will use Corollary~\ref{cor-levy}.  The conditions imposed on $\bf Z$ in Corollary \ref{cor-levy} are satisfied due to the fact that   $\int_{\abs{x}\leq 1} \abs{x}\,\rho(dx)=\infty$ and $\int_{\abs{x}>1}\abs{x}^2\,\rho(dx)<\infty$.  Moreover,  using the asymptotics of the incomplete gamma functions we have that 
\begin{equation}\label{tem-ex-levy23}
\int_\R \big(\abs{xu}\wedge \abs{xu}^2\big)\,\rho(dx)\sim \begin{cases} C_1 u^\alpha & 
\text{as } u\to \infty\\ C_2 u^2 & \text{as } u\to 0
\end{cases}
\end{equation}
where $C_1, C_2>0$ are finite constants depending only on $\alpha, c$ and $\lambda$, and we write $f(u)\sim g(u)$ as $u\to \infty$ (resp.\ $u\to 0$) when $f(u)/g(u)\to 1$ as $u\to \infty$ (resp.\ $u\to 0$).  Eq.~\eqref{tem-ex-levy23} shows that  \eqref{levy-int-cor} is equivalent to \eqref{eq-temp-stable}, and hence the result follows by Corollary~\ref{cor-levy}.   
\end{example}

\medskip
\begin{example}[Multi-stable] \label{cor-stable}
In this example we  extend Example~\ref{ex-stable}(i) to the so called multi-stable processes, that is,  we will consider   $\bf X$  given by \eqref{eq:mma} with  
\begin{equation}
\rho_v(dx)=c(v)\abs{x}^{-\alpha(v)-1} \, dx
\end{equation}
 where $\morf{\alpha}{V}{(0,2)}$ and $\morf{c}{V}{(0,\infty)}$ are measurable functions, and $b=\sigma^2=0$. For $v\in V$,  $\rho_v$ is the L\'evy measure of a symmetric stable distribution with index $\alpha(v)$.  Assume that 
 there exists an $r>1$ such that $\alpha(v)\geq r$ for all $v\in V$. Then  $\X$ is a semimartingale with respect to $\Fi^\Lambda$  if and only if  for $m$-a.e.\ $v$, $f(\cdot,v)$ is absolutely continuous on $[0,\infty)$ with  a derivative $\dot f(\cdot,v)$ satisfying
 \begin{equation}\label{stable-ex}
 \int_V\int_0^\infty\Big( \frac{c(v)}{2-\alpha(v)}\abs{\dot f(s,v)}^{\alpha(v)}\Big) \,ds\,m(dv)<\infty.
 \end{equation}

To show the above  we will argue similarly as in Example~\ref{ex-stable}. By the symmetry, \eqref{drift-4} is satisfied. For all $v\in V$,  $\int_{\abs{x}\leq 1} \abs{x}\,\rho_v(dx)=\infty$, which shows that  \eqref{invar-con} of Theorem~\ref{thm-nes} is satisfied.  
By  basic calculus we have for $v\in V$ that 
\begin{equation}\label{int-stable-24}
u\int_{\abs{x}>u} \abs{x}\,\rho_v(dx)=K(v) \int_{\abs{x}\leq u} x^2\,\rho_v(dx)\end{equation}
where $K(v)=(2-\alpha(v))/(\alpha(v)-1)$. Since $\alpha(v)\geq r$ we have that 
$K(v)\leq 2/(r-1)<\infty$ which together with \eqref{int-stable-24} implies \eqref{eq:u00}. 
From  \eqref{cal-stable} we infer that \eqref{Cf} is equivalent to \eqref{stable-ex}, and thus Theorems~\ref{thm-suf}--\ref{thm-nes}  conclude the proof. 
   \end{example}
    \medskip

  

\begin{example}[supFLP]\label{cor:frac_levy}

 Consider  $\X=\p X$ of the form
\begin{equation}\label{def_multi_frac}
 X_t=\int_{\R\times V} \big( (t-s)_+^{\gamma(v)}-(-s)_+^{\gamma(v)}\big) \,\Lambda(ds,dv),
\end{equation}
where $\morf{\gamma}{V}{(0,\infty)}$ is a measurable function.
 Processes of the form \eqref{def_multi_frac} may be viewed as superpositions of fractional L\'evy processes with (possible) different indexes;  hence the name  supFLP.    If $m$-almost everywhere  $\gamma\in (0,\frac{1}{2})$, $\sigma^2 =0$ and 
  \begin{equation}\label{suf_con_frac}
 \int_V \Big(\int_\R \abs{x}^{\frac{1}{1-\gamma(v)}} \,\rho_v(dx)\Big) \big(\tfrac{1}{2}-\gamma(v)\big)^{-1}\,m(dv)<\infty,
\end{equation}
then  $\X$ is a semimartingale with respect to $\Fi^{\Lambda}$.
Conversely, if $\X$ is a semimartingale with respect to $\Fi^{\Lambda}$ and $\int_{\abs{x}\leq 1}\abs{x}\,\rho_v(dx)=\infty$ for $m$-a.e.\ $v$, then  $m$-a.e.\  $\gamma\in (0,\frac{1}{2})$, $ \sigma^2=0$ and 
\begin{equation}\label{eq:nece_ex}
 \int_\R \abs{x}^{\frac{1}{1-\gamma(v)}}\,\rho_v(dx)<\infty,
  \end{equation}
 and if in addition $\rho$ satisfies \eqref{eq:u00},  then \eqref{suf_con_frac} holds.     
  
  To show the above  let $f(t,v)=t^{\gamma(v)}_+$ for $t\in\R, v\in V$.
  Since   $f(0,v)=0$ for all $v$, \eqref{drift-4} is satisfied. As in Example~\ref{fLp}, we observe that the conditions
 \begin{equation}\label{eq-mult-32}
\int_{\abs{x}\geq 1} \abs{x}^{\frac{1}{1-\gamma(v)}}\,\rho_v(dx)<\infty\quad  \text{and} \quad \gamma(v)<\tfrac{1}{2}\quad  m\text{-a.e.}
 \end{equation}
 follow from the fact  that  $\X$ is a well-defined.  For $\gamma(v)\in (0,\frac{1}{2})$, $f(\cdot,v)$ is absolutely continuous on $[0,\infty)$. By \eqref{basic-cal}  we deduce that 
    \begin{align}\label{est-frac-1}
 \frac{c\abs{x}^{\frac{1}{1-\gamma(v)}}}{\tfrac{1}{2}-\gamma(v)}\leq \int_0^\infty \{\abs{x \dot f(t,v)}\wedge \abs{x \dot f(t,v)}^2\}\,dt \leq   
 \frac{\tilde c\abs{x}^{\frac{1}{1-\gamma(v)}}}{\tfrac{1}{2}-\gamma(v)}
   \end{align}
   for all $x\in \R$, 
where $c, \tilde c >0$ are finite constants not depending $v$ and $x$.  


 By  Theorem~\ref{thm-suf} and \eqref{est-frac-1}, the sufficient part follows. To show the necessary part assume that $\X$ is a semimartingale with respect to $\Fi^{\Lambda}$ and that $\int_{\abs{x}\leq 1}\abs{x}\,\rho_v(dx)=\infty$ for $m$-a.e.\ $v$. 
  By Theorem~\ref{thm-nes},  $f(\cdot,v)$ is absolutely continuous with a derivative $\dot f(\cdot,v)$ satisfying \eqref{int_1} and \eqref{trunc_case}. From \eqref{int_1} we deduce that $\sigma^2=0$ $m$-a.e.\ and  from \eqref{trunc_case} and \eqref{est-frac-1} we infer that  
 \begin{equation}\label{eq-984}
\int_{\abs{x}\leq 1} \abs{x}^{\frac{1}{1-\gamma(v)}}\,\rho_v(dx)<\infty\quad m\text{-a.e.\ }v. 
 \end{equation}
    By   \eqref{eq-mult-32}--\eqref{eq-984}, condition \eqref{eq:nece_ex} follows. Moreover, if  $\rho$ satisfies \eqref{eq:u00},  then  Theorem~\ref{thm-nes} together with  \eqref{est-frac-1} show \eqref{suf_con_frac}. This completes the proof. 
    \end{example}

\newcommand{\bib}{
  \begingroup
  \renewcommand\bibname{References}
   \bibliographystyle{chicago}
    \let\oldthebibliography=\thebibliography
  \renewcommand\thebibliography[1]{
    \oldthebibliography{##1}
    \small
    \setlength\itemsep{3pt plus 3pt minus 2pt}
  }
  \begin{thebibliography}{}

\bibitem[\protect\citeauthoryear{Basse}{Basse}{2008}]{Andreas3}
Basse, A. (2008).
\newblock Gaussian moving averages and semimartingales.
\newblock {\em Electron. J. Probab.\/}~{\em 13\/}(39), 1140--1165.

\bibitem[\protect\citeauthoryear{Basse}{Basse}{2009}]{Andreas2}
Basse, A. (2009).
\newblock Spectral representation of {Gaussian} semimartingales.
\newblock {\em J. Theoret. Probab.\/}~{\em 22\/}(4), 811--826.

\bibitem[\protect\citeauthoryear{Basse and Pedersen}{Basse and
  Pedersen}{2009}]{Basse_Pedersen}
Basse, A. and J.~Pedersen (2009).
\newblock L{\'e}vy driven moving averages and semimartingales.
\newblock {\em Stochastic Process. Appl.\/}~{\em 119\/}(9), 2970--2991.

\bibitem[\protect\citeauthoryear{Basse-O'Connor}{Basse-O'Connor}{2010}]{Andreas1}
Basse-O'Connor, A. (2010).
\newblock Representation of {Gaussian} semimartingales with application to the
  covariance function.
\newblock {\em Stochastics\/}~{\em 82\/}(4), 381--401.

\bibitem[\protect\citeauthoryear{Basse-O'Connor and Graversen}{Basse-O'Connor
  and Graversen}{2010}]{Basse_Graversen}
Basse-O'Connor, A. and S.-E. Graversen (2010).
\newblock Path and semimartingale properties of chaos processes.
\newblock {\em Stochastic Process. Appl.\/}~{\em 120\/}(4), 522--540.

\bibitem[\protect\citeauthoryear{Basse-O'Connor and
  Rosi{\'n}ski}{Basse-O'Connor and Rosi{\'n}ski}{2012a}]{fv-mma}
Basse-O'Connor, A. and J.~Rosi{\'n}ski (2012a).
\newblock Characterization of the finite variation property for a class of
  stationary increment infinitely divisible processes.
\newblock {\em arXiv:1201.4366v1\/}.

\bibitem[\protect\citeauthoryear{Basse-O'Connor and
  Rosi{\'n}ski}{Basse-O'Connor and Rosi{\'n}ski}{2012b}]{Ito-Nisio-D}
Basse-O'Connor, A. and J.~Rosi{\'n}ski (2012b).
\newblock On the uniform convergence of random series in {Skorohod} space and
  representations of c\`adl\`ag infinitely divisible processes.
\newblock {\em Ann. Probab. (To appear)\/}.
\newblock Available at {\small \url{www.imstat.org/aop/future_papers.htm}}.

\bibitem[\protect\citeauthoryear{Beiglb{\"o}ck, Schachermayer, and
  Veliyev}{Beiglb{\"o}ck et~al.}{2011}]{B-D-direct-proof}
Beiglb{\"o}ck, M., W.~Schachermayer, and B.~Veliyev (2011).
\newblock A direct proof of the {B}ichteler-{D}ellacherie theorem and
  connections to arbitrage.
\newblock {\em Ann. Probab.\/}~{\em 39\/}(6), 2424--2440.

\bibitem[\protect\citeauthoryear{Bender, Lindner, and Schicks}{Bender
  et~al.}{2012}]{Be_Li_Sc}
Bender, C., A.~Lindner, and M.~Schicks (2012).
\newblock Finite {V}ariation of {F}ractional {L}\'evy {P}rocesses.
\newblock {\em J. Theoret. Probab.\/}~{\em 25\/}(2), 594--612.

\bibitem[\protect\citeauthoryear{Bichteler}{Bichteler}{1981}]{Bichteler}
Bichteler, K. (1981).
\newblock Stochastic integration and {$L\sp{p}$}-theory of semimartingales.
\newblock {\em Ann. Probab.\/}~{\em 9\/}(1), 49--89.

\bibitem[\protect\citeauthoryear{Bingham, Goldie, and Teugels}{Bingham
  et~al.}{1989}]{Bingham}
Bingham, N.~H., C.~M. Goldie, and J.~L. Teugels (1989).
\newblock {\em Regular Variation}, Volume~27 of {\em Encyclopedia of
  Mathematics and its Applications}.
\newblock Cambridge: Cambridge University Press.

\bibitem[\protect\citeauthoryear{Cheri\-dito}{Cheri\-dito}{2004}]{Patrick}
Cheri\-dito, P. (2004).
\newblock Gaussian moving averages, semimartingales and option pricing.
\newblock {\em Stochastic Process. Appl.\/}~{\em 109\/}(1), 47--68.

\bibitem[\protect\citeauthoryear{Cherny}{Cherny}{2001}]{Cherny}
Cherny, A. (2001).
\newblock When is a moving average a semimartingale?
\newblock {\em MaPhySto -- Research Report\/}~{\em 2001--28}.

\bibitem[\protect\citeauthoryear{Embrechts, Goldie, and Veraverbeke}{Embrechts
  et~al.}{1979}]{reg-var}
Embrechts, P., C.~M. Goldie, and N.~Veraverbeke (1979).
\newblock Subexponentiality and infinite divisibility.
\newblock {\em Z. Wahrsch. Verw. Gebiete\/}~{\em 49\/}(3), 335--347.

\bibitem[\protect\citeauthoryear{Emery}{Emery}{1982}]{Emery}
Emery, M. (1982).
\newblock Covariance des semimartingales gaussiennes.
\newblock {\em C. R. Acad. Sci. Paris S\'er. I Math.\/}~{\em 295\/}(12),
  703--705.

\bibitem[\protect\citeauthoryear{{Gal'chuk}}{{Gal'chuk}}{1985}]{Galchuk}
{Gal'chuk}, L.~I. (1985).
\newblock Gaussian semimartingales.
\newblock In {\em Statistics and Control of Stochastic Processes ({M}oscow,
  1984)}, Transl. Ser. Math. Engrg., pp.\  102--121. New York: Optimization
  Software.

\bibitem[\protect\citeauthoryear{Jacod and Shiryaev}{Jacod and
  Shiryaev}{2003}]{Jacod_S}
Jacod, J. and A.~N. Shiryaev (2003).
\newblock {\em Limit Theorems for Stochastic Processes\/} (Second ed.), Volume
  288 of {\em Grundlehren der Mathematischen Wissenschaften [Fundamental
  Principles of Mathematical Sciences]}.
\newblock Berlin: Springer-Verlag.

\bibitem[\protect\citeauthoryear{Jain and Monrad}{Jain and
  Monrad}{1982}]{Gau_Qua}
Jain, N.~C. and D.~Monrad (1982).
\newblock Gaussian quasimartingales.
\newblock {\em Z. Wahrsch. Verw. Gebiete\/}~{\em 59\/}(2), 139--159.

\bibitem[\protect\citeauthoryear{Jeulin and Yor}{Jeulin and
  Yor}{1993}]{Yor:sem}
Jeulin, T. and M.~Yor (1993).
\newblock Moyennes mobiles et semimartingales.
\newblock In {\em S\'eminaire de {P}robabilit\'es, {XXVII}}, Volume 1557 of
  {\em Lecture Notes in Math.}, pp.\  53--77. Berlin: Springer.

\bibitem[\protect\citeauthoryear{Kardaras and Platen}{Kardaras and
  Platen}{2011}]{K-P}
Kardaras, C. and E.~Platen (2011).
\newblock On the semimartingale property of discounted asset-price processes.
\newblock {\em Stochastic Process. Appl.\/}~{\em 121\/}(11), 2678--2691.

\bibitem[\protect\citeauthoryear{Knight}{Knight}{1992}]{Knight}
Knight, F.~B. (1992).
\newblock {\em Foundations of the Prediction Process}, Volume~1 of {\em Oxford
  Studies in Probability}.
\newblock New York: The Clarendon Press Oxford University Press.
\newblock Oxford Science Publications.

\bibitem[\protect\citeauthoryear{Liptser and Shiryayev}{Liptser and
  Shiryayev}{1989}]{Lip_S}
Liptser, R.~S. and A.~Shiryayev (1989).
\newblock {\em Theory of Martingales}, Volume~49 of {\em Mathematics and its
  Applications (Soviet Series)}.
\newblock Dordrecht: Kluwer Academic Publishers Group.

\bibitem[\protect\citeauthoryear{Meyer}{Meyer}{1984}]{Meyer_app}
Meyer, P.-A. (1984).
\newblock Un r\'esultat d'approximation.
\newblock In {\em {S\'eminaire} de {Probabilit\'es}, {XVIII}}, Volume 1059 of
  {\em Lecture Notes in Math.}, pp.\  268--270. Berlin: Springer.

\bibitem[\protect\citeauthoryear{Protter}{Protter}{2004}]{Protter}
Protter, P.~E. (2004).
\newblock {\em Stochastic Integration and Differential Equations\/} (Second
  ed.), Volume~21 of {\em Applications of Mathematics (New York)}.
\newblock Berlin: Springer-Verlag.
\newblock Stochastic Modelling and Applied Probability.

\bibitem[\protect\citeauthoryear{Rajput and Rosi{\'n}ski}{Rajput and
  Rosi{\'n}ski}{1989}]{Rosinski_spec}
Rajput, B.~S. and J.~Rosi{\'n}ski (1989).
\newblock Spectral representations of infinitely divisible processes.
\newblock {\em Probab. Theory Related Fields\/}~{\em 82\/}(3), 451--487.

\bibitem[\protect\citeauthoryear{Rosi{\'n}ski}{Rosi{\'n}ski}{2001}]{Rosinski_series_point}
Rosi{\'n}ski, J. (2001).
\newblock Series representations of {L}\'evy processes from the perspective of
  point processes.
\newblock In {\em L\'evy Processes}, pp.\  401--415. Boston, MA: Birkh\"auser
  Boston.

\bibitem[\protect\citeauthoryear{Sato}{Sato}{1999}]{Sato}
Sato, K. (1999).
\newblock {\em L\'evy Processes and Infinitely Divisible Distributions},
  Volume~68 of {\em Cambridge Studies in Advanced Mathematics}.
\newblock Cambridge: Cambridge University Press.
\newblock Translated from the 1990 Japanese original, Revised by the author.

\bibitem[\protect\citeauthoryear{Stricker}{Stricker}{1983}]{Stricker_gl}
Stricker, C. (1983).
\newblock Semimartingales gaussiennes---application au probl\`eme de
  l'innovation.
\newblock {\em Z. Wahrsch. Verw. Gebiete\/}~{\em 64\/}(3), 303--312.

\bibitem[\protect\citeauthoryear{Surgailis, Rosi{\'n}ski, Mandrekar, and
  Cambanis}{Surgailis et~al.}{1993}]{mixed_ma_R}
Surgailis, D., J.~Rosi{\'n}ski, V.~Mandrekar, and S.~Cambanis (1993).
\newblock Stable mixed moving averages.
\newblock {\em Probab. Theory Relat. Fields\/}~{\em 97\/}(4), 543--558.

\end{thebibliography}
  \endgroup
}

\bib

\end{document}